\documentclass[a4paper,11pt, reqno]{amsart}

\usepackage{amsthm}
\usepackage{amssymb}
\usepackage{amsmath}
\usepackage{mathtools}
\usepackage{pdfpages}
\usepackage{changes}
\usepackage{float}
\usepackage{tikz-cd}
\usepackage{exscale,cite,color,amsopn}
\usepackage[colorlinks=true,pdftex,unicode=true,linktocpage,bookmarksopen,hypertexnames=false]{hyperref}
\usepackage{colortbl}

\renewcommand{\leq}{\leqslant}
\renewcommand{\geq}{\geqslant}

\DeclareMathOperator{\Fix}{Fix}
\DeclareMathOperator{\id}{id}

\DeclareMathOperator{\Ker}{Ker}
\DeclareMathOperator{\Rad}{Rad}

\DeclareMathOperator{\Soc}{Soc}
\DeclareMathOperator{\Sym}{Sym}
\DeclareMathOperator{\Ann}{Ann}

\newcommand{\Z}{\mathbb{Z}}

\newcommand{\Aut}{\operatorname{Aut}}

\makeatletter
\numberwithin{equation}{section}
\numberwithin{figure}{section}
\numberwithin{table}{section}
\newtheorem{thm}{Theorem}[section]
\newtheorem*{thm*}{Theorem}
\newtheorem{lem}[thm]{Lemma}
\newtheorem{cor}[thm]{Corollary}
\newtheorem{pro}[thm]{Proposition}

\theoremstyle{definition} 
\newtheorem{defn}[thm]{Definition}

\newtheorem{corollary}[thm]{Corollary}

\newtheorem{exa}[thm]{Example}

\makeatother

\title[]{Nilpotency of skew braces and multipermutation solutions of the Yang-Baxter equation}

\author{E. Jespers \and A. Van Antwerpen \and L. Vendramin}
\address{Department of Mathematics and Data Science, Vrije Universiteit Brussel, Pleinlaan 2, 1050 Brussel, Belgium}
\email{Eric.Jespers@vub.be}
\email{Arne.Van.Antwerpen@vub.be}
\email{Leandro.Vendramin@vub.be}


\begin{document}

\begin{abstract}
We study relations between different notions of nilpotency in the
context of skew braces and applications to the structure of 
solutions to the Yang--Baxter equation. In particular, 
we consider annihilator nilpotent skew braces, an important
class that turns out to be a brace-theoretic analog to the class
of nilpotent groups. In this vein, several well-known theorems 
in group theory are proved in the more general setting of skew braces.
\end{abstract}
 
\maketitle
\section{Introduction}

The Yang--Baxter equation (YBE) is a fundamental object in pure mathematics and physics. It arose from the work of the Nobel prize-winning 
physicist Yang \cite{Ya} on statistical mechanics and independently 
in the work of Baxter \cite{Ba} concerning the 8-vertex model. A number of exciting developments in physics and mathematics have led to the conclusion that the YBE is a fundamental mathematical structure, with connections to various topics such as knot theory, 
braid theory, operator theory, Hopf algebras, quantum groups, 3-manifolds and the monodromy of differential equations.

A solution to the YBE is a pair $(V,R)$, where
$V$ is a vector space and $R\colon V\otimes V\to V\otimes V$ is a
linear map such that 
\[
(R\otimes\id)(\id\otimes R)(R\otimes\id)=(\id\otimes R)(R\otimes\id)(\id\otimes R).
\]
An absolutely fascinating family of solutions consist of those maps $R$ 
induced by linear extensions of a bijective map 
$r\colon X\times X\to X\times X$, 
where $X$ is a basis of $V$. 
As this idea captures all the combinatorics behind the YBE, 
one says that the pair $(X,r)$ is a combinatorial solution, also called set-theoretic solution by Drinfel'd \cite{MR1183474}.

The fundamental problem of constructing combinatorial solutions 
with given properties is nowadays based on the use of certain (associative and non-associative) algebraic structures acting on the solution. Examples of such structures are 
groups of I-type \cite{MR1637256}, bijective 1-cocycles \cite{MR1769723,MR1769723,MR1809284}, skew braces \cite{MR3177933,MR2278047,MR3647970} and cycle sets \cite{MR2132760}. 
Properties of solutions can be described in terms of such structures. For that reason, the algebra behind 
the YBE is now the focus of intensive research. One particular algebraic structure stands out: skew braces, whose theory has its origins in Jacobson radical rings, and now appears in several 
different areas of mathematics \cite{MR3291816,MR3763907}. These algebraic structures cover the class of bijective 
non-degenerate combinatorial solutions, i.e. bijective solutions $(X,r)$ such that if one writes 
\[
r(x,y)=(\sigma_x(y),\tau_y(x))
\]
for any $x,y \in X$, then all the maps  $\sigma_x$ and $\tau_y$ are bijective. 
For simplicity's sake we call bijective non-degenerate solutions just solutions throughout the paper.

A first source of inspiration for this paper is the work  
of Diaconis about properties of random algebraic objects \cite{MR1159288}. 
Our motivation is the following
question: Which properties can be found in a randomly chosen solution to the YBE?  

A solution $(X,r)$ is said to be \emph{multipermutation} if $(X,r)$ 
can be retracted into the trivial solution 
over the singleton after finitely many steps of identifying elements $x$ and $y$ with $\sigma_x =\sigma_y$ and $\tau_x =\tau_y$. 
Our first goal is to \emph{provide a rigorous framework} 
for the following experimental observation: \emph{almost all} 
finite solutions are multipermutation solutions. 
In fact, among all involutive 
solutions with at most ten elements, more than 95\% are
multipermutation solutions. 

The importance of multipermutation solutions was first recognized  
in the seminal work of Etingof, Schedler and Soloviev \cite{MR1722951} for involutive solutions. This was
later confirmed by several existing connections 
between multipermutation solutions
and other topics such as left-ordered groups and right nilpotency skew braces, see for example 
\cite{MR3177933,MR2652212,MR3861714,MR2885602,MR3814340}.
 
At the moment, we are far from being able to 
understand properties of randomly chosen solutions to the YBE. However,
to provide a rigorous framework 
that proves that almost every solution is a multipermutation solution, 
it is somewhat clear that this will be achieved 
studying nilpotency of skew braces.

There are two skew braces that control the structure of a non-degenerate bijective solution $(X,r)$. 
The first one is the structure skew brace $G(X,r)$. The multiplicative
group of this brace is isomorphic to the group with generators
in $X$ and relations $xy=\sigma_x(y)\tau_y(x)$ for $x,y\in X$.
The additive
group of $G(X,r)$ is isomorphic to the group 
with generators $X$ and relations $y^{-1}xy=\sigma_y \tau_{\sigma^{-1}_x (y) }(x)$.

To study multipermutation solutions right nilpotency 
was introduced by Rump  \cite{MR2278047} for skew braces of abelian type (that is, the additive group of the respective  skew brace is abelian) and later
extended to arbitrary skew braces in \cite{MR3763907,MR3957824}.
A combination of results of Ced\'{o}, Smoktunowicz and Vendramin \cite[Theorem 2.20]{MR3957824} 
and Ced\'{o}, Jespers, Kubat, Van Antwerpen and Verwimp \cite[Theorem 4.13]{cedo2020various} 
proved the following equivalence  for solutions:
\begin{align*}
(X,r)\text{ \it is multipermutation}
&\Longleftrightarrow
G(X,r)\text{ \it is right nilpotent of nilpotent type},
\end{align*}
where being of nilpotent type means that the additive structure of the respective skew brace is nilpotent. This was shown by Gateva--Ivanova \cite{MR3861714} for involutive solutions.
The second skew brace associated to a non-degenerate bijective solution is the permutation skew brace $\mathcal{G}(X,r)$, i.e. 
the subgroup of $\Sym (X) \times \Sym (X)$  generated by the elements $(\sigma_x,\tau_x^{-1})$. This skew brace is a quotient of $G(X,r)$ by an ideal contained in 
its socle $\Soc(G(X,r))$. 
Moreover, using Proposition 2.19 in \cite{MR3957824}, one obtains the characterization $$(X,r)\text{ \it is multipermutation} \Longleftrightarrow \mathcal{G}(X,r)\text{ \it is right nilpotent of nilpotent type}.$$ In particular, if $(X,r)$ is finite, this characterization is realized by the  finite skew brace $\mathcal{G}(X,r)$.

By analogy, Rump \cite{MR2278047}, and Ced\'{o}, Smoktunowicz and Vendramin \cite{MR3957824} 
also defined left nilpotency of a skew brace, where the latter authors showed that left nilpotency of skew braces of nilpotent type is equivalent to the multiplicative group being nilpotent.
Left and right $p$-nilpotency of finite
skew braces were studied in \cite{MR4062375,MR3935814}. 
Recently, motivated by universal algebra, 
Bonatto and Jedli{\v c}ka \cite{BonJed2021} introduced
the notion of central nilpotency, which yields a subclass of left and right nilpotent skew braces. As this class of skew braces is defined through iterative use of the annihilator, we call this annihilator nilpotent.

Computer experiments show that among the 36116 involutive multipermutation solutions of size at most eight 34770 of them admit a left nilpotent $\mathcal{G}(X,r)$ permutation brace. This implies that at least for solutions of size at most $8$ almost all of their permutation braces are annihilator nilpotent. 

There are several reasons to consider annihilator nilpotent skew braces. On the one hand, the previous observation suggests that the study of
annihilator nilpotent skew braces is the first step towards 
the problem of finding a precise framework that allows us to prove
that almost all solutions are multipermutation solutions. On the other hand,
this structural property is extremely important for the general algebraic 
theory of skew braces, as annihilator nilpotent skew braces are to skew 
braces what nilpotent groups are for groups.

\medskip
The paper is organized as follows. In Section \ref{preliminaries}
we discuss definitions and basic properties of skew braces. We introduce
annihilator nilpotent skew braces and present several examples. 
This section also includes several 
characterizations of anihilator nilpotency and results regarding
connections between this
and similar nilpotency notions. In Section \ref{noetherian} we study
annihilator nilpotent skew braces with ascending chain condition (ACC) on sub skew braces. 
It turns out that  an annihilator nilpotent skew brace
has ACC on sub skew braces 
if and only if it is finitely generated as a skew brace. 
In Section \ref{torsion}
we study the relationship between additive and multiplicative torsion 
of annihilator nilpotent skew braces. Finally, in Section \ref{generating}
we prove a brace-theoretic analog of a result of 
Hirsch about non-generating elements (Theorem \ref{thm:HirschBraceNilp})
and study its consequences. 

\section{Definitions and preliminaries}
\label{preliminaries}
In this section we fix notation and introduce some crucial definitions and lemmas.
A \emph{skew brace} is a triple $(B,+,\circ)$, where $(B,+)$ and $(B,\circ)$ are
groups (called the additive and multiplicative group respectively) and 
\[
a\circ (b+c)=a\circ b-a+a\circ c
\]
holds for all $a,b,c\in B$. If $B$ 
is a skew brace and $a\in B$, then $\overline{a}$ denotes the inverse of 
$a$ with respect to the multiplicative operation. 
For  a skew brace $B$ the map $\lambda\colon (B,\circ)\to\Aut(B,+)$, $a\mapsto \lambda_a$, with 
$\lambda_a(b)=-a+a\circ b$, is a group homomorphism. 

A skew brace is said to be \emph{trivial} if both operations coincide. 

We regularly make use of the following lemma, see for example \cite{MR4023387}.
Recall that for elements $a$ and $b$ of a 
skew left brace $B$ one  puts 
\[
a * b = \lambda_a(b)-b=-a+a\circ b-b.
\]
This operation $*$ measures the ``difference'' between the additive 
and multiplicative operations. The center of a group $(G,\cdot)$ is denoted $Z(G)$ or $Z(G,\cdot)$ to emphasize the operation.

\begin{lem} 
\label{lem:id*}
 Let $B$ be a skew brace. For any $x,y,z \in B$  the following properties hold:
 \begin{enumerate}
 \item  $x*(y+z)=x*y +y +x*z-y$,
 \item $(x+ y)*z=x*(\lambda_x^{-1}(y) * z)+\lambda_x^{-1}(y)*z +x*z$,
 \item $(x \circ y) *z = x*(y*z) + y*z + x*z$.
 \end{enumerate}
\end{lem}

\begin{defn}
Let $B$ be a skew brace. Recall that the socle  of $B$, denoted $\Soc(B)$, is defined as $$ \Soc(B) = \Ker \lambda \cap Z(B,+).$$ Moreover, the annihilator $\Ann(B)$ of $B$ is defined as $$ \Ann(B) = \Soc(B) \cap \Fix(B),$$
where $\Fix(B)=\{x\in B:\lambda_a(x)=x\text{ for all $a\in B$}\}$. 
\end{defn}

Note that 
\[
\Ann(B)=\{ x\in B : x\circ a =a\circ x =x+a =a+x, \mbox{ for all } x\in B\}.
\]
Similarly to the upper central series in nilpotent groups, one can define the $n$-th annihilator of a skew brace. If $\Ann_n(B)$ is defined, then denote $\pi:B \longrightarrow B/\Ann_n(B)$ the canonical morphism. One denotes
\[
\Ann_{n+1}(B) = \pi^{-1}(\Ann(B/\Ann_n(B))).
\]
By the isomorphism theorems, these are all ideals in $B$. 

In \cite{BonJed2021}, Bonatto and Jedli{\v c}ka introduced the following notion of nilpotency. Note that they call this centrally nilpotent, as the annihilator of a skew brace corresponds to a notion of center coming from universal algebra. However, as this would further mystify the terminology, we opt to call this notion annihilator nilpotent.

\begin{defn}
Let $B$ be a skew brace. An ascending series 
\[
0=I_0 \subseteq I_1 \subseteq ... \subseteq I_n = B
\]
of ideals of $B$ is called an annihilator series of $B$, if $I_{j+1}/I_j \subseteq \Ann(B/I_{j})$ for $0\leq j <n$. If $B$ has an annihilator series, then $B$ is called annihilator nilpotent. Clearly, $B$ is annihilator nilpotent if and only if $\Ann_n(B)=B$ for some positive integer $n$.
\end{defn}

Obviously, sub skew braces and epimorphic images of annihilator nilpotent skew braces also are annihilator skew braces. Note that if $B$ 
is an annihilator nilpotent skew brace, 
then both $(B,+)$ and $(B,\circ)$ are nilpotent groups.

\begin{exa}
Let $(G,\circ)$ be a group and denote by $\circ^{opp}$ the opposite operation of $\circ$. Then $(G,\circ,\circ^{opp})$ is annihilator nilpotent if and only if the group 
$(G,\circ)$ is nilpotent. In particular, every result on annihilator nilpotent skew braces recovers a similar result on nilpotent groups.
\end{exa}

Recall that a skew brace $B$ is said to be \emph{two-sided} if 
\[
(a+b)\circ c=a\circ c-c+b\circ c
\]
holds for all $a,b,c\in B$. 

\begin{exa}
It was shown in \cite{MR2278047} that a Jacobson radical ring 
$(B,+,*)$ corresponds to a two-sided brace $B$, where 
$a\circ b = a+ a*b +b$. Moreover, ring theoretic ideals of $B$ 
coincide with brace theoretic ideals of $B$. 
Hence nilpotent Jacobson radical rings (for example finite Jacobson radical rings) are annihilator nilpotent as braces. 
To prove this one makes use of the following fact:  if the $*$-product 
of $n$ factors of $B$ is  zero, then the $*$-product of $n-1$ factors 
of $B$ is contained  in $\Ann(B)$ as brace. 
\end{exa}

\begin{exa}
The operation $x\circ y=x+y+2xy$, using the ring structure on $\Z/4$, induces an annihilator nilpotent skew brace $(\Z/4,+,\circ)$ with multiplicative group
isomorphic to $\Z/2\times\Z/2$ and additive group $\Z/4$.
\end{exa}

\begin{exa}
Of the $47$ skew braces of order $8$, there exist only $2$ skew braces that are not annihilator nilpotent, which even turn out to be of abelian type. This indicates that they dominate skew $p$-braces. Note that only $3$ of these are not two-sided. Of order $16$ only $40$ of $1605$ skew braces are not annihilator nilpotent. Of order $27$ only $4$ of the $101$ are not annihilator nilpotent of which $27$ are not two-sided.
\end{exa}

We now link the notion of annihilator nilpotent skew braces to the known notions of nilpotency. For this we first recall these. The following notions of left and right nilpotent were defined by Rump \cite{MR2278047}. It was shown \cite{MR3814340,MR3957824} that left nilpotency coincides for finite skew braces of nilpotent type
(i.e. the additive group is nilpotent) to the nilpotency of the multiplicative group. Moreover, it was shown \cite{MR3957824} that such a finite skew brace is the direct product of braces of prime power order.

Let $B$ be a skew brace and let $X$ and $Y$ be subsets of $B$. Then $X*Y$ is, by definition, 
the additive subgroup of $B$ generated by $\{x*y:x\in X,\,y\in Y\}$.

\begin{defn}
Let $B$ be a skew brace. Denote $B^{n+1} = B*B^n$ with $B^1=B$ for any positive integer $n$. 
Then the skew brace $B$ is said to be left nilpotent if $B^n=0$ for some positive integer $n$.
\end{defn}

On the other hand, as mentioned in the introduction, right nilpotency is related to 
multipermutation solutions. In fact, a solution $(X,r)$ is multipermutation if and only if $G(X,r)$ is a right nilpotent skew brace of nilpotent type. 


\begin{defn}
Let $B$ be a skew brace. Denote $B^{(n+1)}=B^{(n)}*B$ with $B^{(1)}=B$. 
Then the skew brace $B$ is said to be right nilpotent if $B^{(n)}=0$ for some positive integer $n$.
\end{defn}

Smoktunowicz introduced \cite{MR3814340} strongly nilpotent braces. 

\begin{defn}
Let $B$ be a skew brace. Denote 
\[
B^{\left[n+1\right]}=\left\langle\bigcup_{i=1}^n B^{[i]}*B^{[n-i]}\right\rangle_+
\]
with $B^{[1]}=B$. Then the skew brace $B$ is said to be strongly nilpotent, if $B^{[n]}=0$ for some positive integer $n$.
\end{defn}

It was shown by Ced\'{o}, Smoktunowicz and Vendramin \cite{MR3957824} (and earlier for left braces by Smoktunowicz) that strongly nilpotent skew braces coincide with skew braces that are both left and right nilpotent.
Hence we have the following link between the different notions of nilpotency of skew braces.




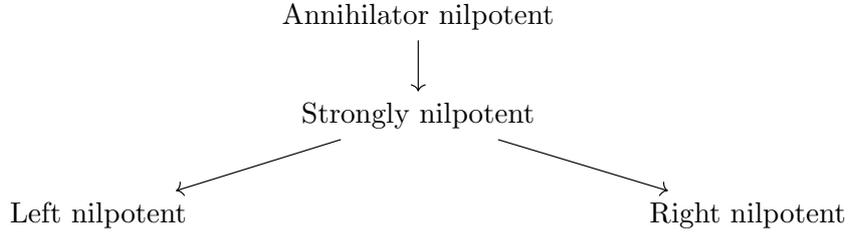
\[\begin{tikzcd}
 & \text{Annihilator nilpotent} \arrow{d} & 
 \\
 & \text{Strongly nilpotent} \arrow{dl} \arrow{dr}& 
 \\
 \text{Left nilpotent} & & \text{Right nilpotent}
\end{tikzcd}\]

\begin{exa}
    Let $X=\{1,\dots,5\}$ and $(X,r)$ be the involutive solution
    given by $\sigma_1=\sigma_2=\sigma_3=\id$, $\sigma_4=(23)(45)$ and
    $\sigma_5=(12)(45)$. Then $(X,r)$ is a multipermutation
    solution. Moreover, $\mathcal{G}(X,r)$ is a skew brace
    of abelian type with additive group isomorphic to $\Z/6$ and
    multiplicative group isomorphic to $\Sym_3$. The skew brace
    $\mathcal{G}(X,r)$ is right nilpotent and not left nilpotent.
\end{exa}

Let $A$ be a skew brace. An \emph{s-series} of $A$ is a sequence 
\[
A=I_0\supseteq I_1\supseteq I_2\supseteq\cdots\supseteq I_n=0 
\]
of ideals of $A$ such that $I_{j-1}/I_j\subseteq\Soc(A/I_j)$ for each $j$.

In \cite[Lemma 2.14]{MR3957824} it is proved that 
if $A=I_0\supseteq I_1\supseteq I_2\supseteq\cdots\supseteq I_n=0$ is an $s$-series for $A$, then $A^{(i+1)}\subseteq I_i$ 
for all i.

Let 
$\Soc_0(A)=\{0\}$ and, for $n\geq 1$, let $\Soc_n(A)$ be the unique
ideal of $A$ such that $\Soc_{n-1}(A)\subseteq\Soc_n(A)$ and 
\[
\Soc_n(A)/\Soc_{n-1}(A)=\Soc(A/\Soc_{n-1}(A)).
\]
Then 
\[
\Soc_0(A)\subseteq\Soc_1(A)\subseteq\cdots\subseteq\Soc_n(A)\subseteq\cdots
\]
is the socle series of $A$. By \cite[Lemma 2.15]{MR3957824}, a skew brace
$A$ admits an s-series if and only if there exists a positive integer $n$ 
such that $A=\Soc_n(A)$.

Let $B$ be a skew brace and $X$ and $Y$ be subsets of $B$. We write $[X,Y]_+$ (resp. $[X,Y]_\circ$) to denote 
the additive (resp. multiplicative) subgroup generated by commutators 
\[
[x,y]_+=x+y-x-y\qquad\text{(resp. $[x,y]_\circ=x\circ y\circ \overline{x}\circ\overline{y}$)}, 
\]
$x\in X$ and $y\in Y$. 

\begin{pro}
\label{pro:leftandrightgivesmult}
Let $B$ be a skew brace of nilpotent type. If $B$ is left and right nilpotent, then $(B,\circ)$ is a nilpotent group.
\end{pro}

\begin{proof}
Let $B$ be left and right nilpotent of nilpotent type. 
By \cite[Lemma 2.15]{MR3957824}, there exists a
positive integer $m$ such that $B=\Soc_m(B)$. We may assume that $k$ is the smallest non-negative integer such that $B=\Soc_k(B)$, so $k$ is the length of 
the socle series of $B$. We proceed by induction 
on $k$.
If $B = \Soc(B)$, the theorem clearly holds.
Suppose we have shown that the theorem holds for skew braces with a socle series of length at most $k$.

Let $B$ be a left and right nilpotent skew brace of nilpotent type with a socle series of length $k+1$.
Denote $\gamma_n(B,\circ)$ the $n$-th term of the lower central series of $(B,\circ)$. Then, applying the induction 
hypothesis on $B/\Soc(B)$, one obtains that $\gamma_m(B,\circ) \subseteq \Soc(B)$ for some positive 
integer $m$. 
Consider $x \in \Soc(B)$ and $b \in B$, then $$ \left\lbrack b,x \right\rbrack_{\circ} = 
b \circ x \circ \overline{b} \circ \overline{x} =
 b + \lambda_{b}(x) - \lambda_{b\circ x \circ \overline{b}}(b) 
 - \lambda_{\left\lbrack b,x \right\rbrack_{\circ}}(x).$$ 
As $\lambda_x = \id_B$  and because $\lambda_b(x) \in Z(B,+)$ 
this becomes $$\left\lbrack b,x \right\rbrack_{\circ} = b + \lambda_b(x) -b-x= b*x.$$ 
As $\Soc(B)$ is a normal subgroup of $(B,\circ)$ and $\gamma_m(B,\circ) \subseteq \Soc(B)$, it follows that for all positive integers 
$l$ it holds that $\gamma_{m+l}(B) \subseteq \Soc(B)$. Hence, applying the formula above, one 
obtains that $\gamma_{m+l}(B) \subseteq B^l$. As $B$ is left nilpotent, this shows that 
the lower central series of $(B,\circ)$ ends in $0$ and thus $(B,\circ)$ is a nilpotent group.
\end{proof}

Inspired by the lower central series of group theory, 
Bonatto and Jedli{\v c}ka \cite{BonJed2021} introduced the following series of ideals.

\begin{defn}
Let $B$ be a skew brace and $I$ an ideal of $B$. Denote $\Gamma_0(I) = I$ and $\Gamma_{n+1}(I) =\left< \Gamma_n(I)*B, B*\Gamma_n(I), [B,\Gamma_n(I)]_+\right>_+$ for $n\geq0$.
\end{defn}

Note that if $B$ is a skew brace and $I$ is an ideal of $B$, then each $\Gamma_n(I)$ is an ideal of $B$.
It was shown by Bonatto and Jedli{\v c}ka \cite{BonJed2021} that the series
  $$\Gamma_0 (B) \supseteq \Gamma_1 (B) \supseteq \Gamma_2 (B) \supseteq \cdots$$
 plays the role of a lower central series for annihilator nilpotency, which we state for further reference.

\begin{thm}[Bonatto--Jedli{\v c}ka]
\label{thm:bonatto-nilpotent}
Let $B$ be a skew brace. Then $B$ is annihilator nilpotent if and only if $\Gamma_c(B) = 0$ for some positive integer $c$.
\end{thm}

\begin{proof}
    See \cite[Theorem 2.7]{BonJed2021}. 
\end{proof}

\begin{corollary}
Let $B$ be a skew brace of nilpotent type. Then the following properties  are equivalent:
\begin{enumerate}
    \item $B$ is annihilator nilpotent.
    \item $B$ is left and right nilpotent.
    \item $B$ is strongly nilpotent.
\end{enumerate}
\end{corollary}

\begin{proof} 
By \cite[Theorem 2.30]{MR3957824}, statements $(2)$ and $(3)$ are equivalent.

That $(1)$ implies $(2)$ follows immediately from Theorem \ref{thm:bonatto-nilpotent}.

To prove $(2)$ implies $(1)$, 
assume  $B$ is left and right nilpotent. We will proceed by induction on the length of the s-series of $B$, which exists 
because of \cite[Lemma 2.16]{MR3957824}. 
By Proposition~\ref{pro:leftandrightgivesmult}, $(B,\circ)$ is a nilpotent group of class $c$. As $\Soc(B)$ is a normal subgroup of $(B,\circ)$, it follows that 
\[
0 \neq \Soc(B)\cap Z(B,\circ) =\Ann(B).
\]
Inductively, one sees that $\Soc(B) \cap Z^k(B,\circ) \subseteq\Ann_k(B)$ for all $k$. 
Hence $\Soc(B) \subseteq \Ann_c(B)$. Since, by the induction hypothesis, $B/\Soc(B)$ has a finite annihilator series, 
it follows that $B$ has a finite annihilator series,  i.e. $B$ is annihilator nilpotent.
\end{proof}

Inspired by the notion of strong nilpotency and a question of Bonatto and Jedli{\v c}ka, we introduce the following chain of ideals. 
\begin{defn}
Let $B$ be a skew brace. Denote $$\Gamma_{[n]}(B)=\left< \Gamma_{[i]}(B)*\Gamma_{[n-i]}(B), \left\lbrack\Gamma_{[i]}(B),\Gamma_{[n-i]}(B)\right\rbrack_+: 1\leq i\leq n-1\right>_+,$$ with $\Gamma_{[1]}(B)=B$.
\end{defn}

\begin{lem}
Let $B$ be a skew brace. Then each $ \Gamma_{\lbrack n\rbrack}(B)$ is an ideal and $\Gamma_{\lbrack n+1 \rbrack}(B) \subseteq \Gamma_{\lbrack n \rbrack}(B)$ for all positive integers $n$.
\end{lem}

\begin{proof}
We will show the result by induction on the index $n$. Clearly, $\Gamma_{[1]}(B)$ is an ideal. Moreover, as $\Gamma_{[2]}(B)=B*B+\left\lbrack B,B\right\rbrack_+$, it is easy to see that $\Gamma_{[2]}(B)/B*B$ is an ideal of $B/B*B$, which shows that $\Gamma_{[2]}(B)$ is an ideal and, clearly, $\Gamma_{[2]}(B)\subseteq \Gamma_{[1]}(B)$.  Suppose that for $k<n$ we have that $\Gamma_{[k]}(B) \subseteq \Gamma_{[k-1]}(B)$ and $\Gamma_{[k]}(B)$ is an ideal. We show the result for $n$. Note that if $k>1$ then, by the induction hypothesis, $$\Gamma_{[k]}(B) * \Gamma_{[n-k]}(B) \subseteq \Gamma_{[k-1]}(B)*\Gamma_{[n-k]}(B) \subseteq \Gamma_{[n-1]}(B).$$ Analogously, one proves that \[
\left\lbrack \Gamma_{[k]}(B),\Gamma_{[n-k]}(B)\right\rbrack_+ \subseteq \Gamma_{[n-1]}(B),
\]
which shows that $\Gamma_{[n]}(B) \subseteq \Gamma_{[n-1]}(B)$. In particular, this entails, by definition of $\Gamma_{[n]}(B)$, that $$ B*\Gamma_{[n]}(B) \subseteq B*\Gamma_{[n-1]}(B) \subseteq \Gamma_{[n]}(B).$$ Analogously, one shows that 
$\Gamma_{[n]}(B)*B\subseteq \Gamma_{[n]}(B)$ and 
\[
\left\lbrack B, \Gamma_{[n]}(B)\right\rbrack_+ \subseteq \Gamma_{[n]}(B).
\]
As $\Gamma_{[n]}(B)$ is an additive subgroup by definition, the result follows.
\end{proof}

The following proposition further substantiates the problem proposed by Bonatto and Jedli{\v c}ka whether all commutators, 
coming from a universal algebra setting, are nested binary commutators and whether the latter are immediately satisfied by annihilator nilpotent skew braces.
Furthermore, it allows us to use the series $\Gamma_{[n]}(B)$ in the rest of the paper.

\begin{thm}\label{thm:gamma}
Let $B$ be a skew brace. Then $B$ is annihilator nilpotent if and only if there exists a positive integer $c$ such that $\Gamma_{[c]}(B) = 0$.
\end{thm}

\begin{proof}
The sufficiency is obvious as $\Gamma_n(B) \subseteq \Gamma_{[n]}(B)$ for all $n$.

We prove necessity by induction. If $\Gamma_{2}(B)=0$, then $B$ is a trivial brace. Hence $\Gamma_{[2]}(B) = 0$, which proves the base case.
Suppose that we have shown that $\Gamma_{\left[2^k\right]}(B)=0$ for all skew braces $B$ with $\Gamma_{k}(B)=0$ for $k\leq n$.

Consider a skew brace $B$ such that $\Gamma_{n+1}(B)=0$. Let $a \in \Gamma_{\left[2^{n+1}\right]}(B)$. 
Then $a$ is a sum of elements of the form $a_i*b_i$ 
and $\left\lbrack a_i, b_i \right\rbrack_+$ with $a_i \in \Gamma_{[q_i]}(B)$ and $b_i\in \Gamma_{\left[2^{n+1}-q_i\right\rbrack}(B)$. Consider the cases where $q_i \geq 2^{n}$. Then applying the induction hypothesis on $B/\Gamma_{n}(B)$, we obtain that $a_i\in \Gamma_n(B)$. 
Hence both $a_i*b_i \in \Gamma_{n+1}(B)=0$ and $\left\lbrack a_i, b_i \right\rbrack_+\in \Gamma_{n+1}(B)=0$, 
showing that these terms are all zero. Remain the cases where $q_i<2^n$. 
However, applying the same reasoning on $b_i \in \Gamma_{[2^{n+1}-q_i]}(B)\subseteq \Gamma_{\left\lbrack 2^{n}\right\rbrack}(B)$ 
we obtain that also all these terms are zero. Hence $a=0$.
\end{proof}

\section{Annihilator nilpotent skew braces with ACC on sub skew braces}
\label{noetherian}

A skew brace is said to have  the ascending chain condition (ACC) on sub skew braces if it satisfies the ascending chain condition on sub skew braces. 
In this section we show that for annihilator nilpotent skew braces this condition is equivalent with being finitely generated as skew braces, or equivalent with its additive (respectively multiplicative) group being finitely generated.
The necessity of these conditions is shown in the following lemma.

\begin{lem}\label{lem:noetherianannnilp}
Let $B$ be a skew brace. If $B$ has ACC on sub skew braces and is annihilator nilpotent, then the following properties are equivalent:
\begin{enumerate}
    \item The skew brace $B$ is finitely generated.
    \item The group $(B,+)$ is finitely generated.
    \item The group $(B,\circ)$ is finitely generated.
\end{enumerate}
\end{lem}
\begin{proof}
Let $B$ be finitely generated. Consider the annihilator series $$ 0 \subseteq \Ann(B) \subseteq ... \subseteq \Ann_k(B) = B.$$ As $B$ has ACC on sub skew braces, the ideal $\Ann(B)$ is finitely generated as a skew brace. As the latter is a trivial brace, this implies that it is both additively and multiplicatively finitely generated. Similarly, 
\[
\Ann_{i+1}(B)/\Ann_{i}(B)
\]
is a finitely generated trivial brace. Hence, every factor of the annihilator series is finitely generated both as additive and multiplicative group, which implies that $B$ is finitely generated as additive and multiplicative group.

Clearly, if $(B,+)$ or $(B,\circ)$ are finitely generated, then $B$ is finitely generated, showing the result.
\end{proof}

In the following lemma, we often use, without mention, the following identity (stated in Lemma~\ref{lem:id*}.(3)) that holds in any skew brace
 $$ (a\circ b) * c = a*(b*c) +b*c +a*c.$$

\begin{lem}\label{lem:gammaminus2}
Let $B$ be an annihilator nilpotent skew brace with $\Gamma_{[c]}(B)=0$. If $a \in \Gamma_{\left[c-k\right]}(B)$ and $q,w \in \Gamma_{[k-1]}$, then the following statements hold:
\begin{enumerate}
\item $a*(q + w) = a*q + a*w$,
\item $a*(q\circ w) = a*q + a*w$,
\item $(q\circ w) *a = q*a + w*a$,
\item $(q+w)*a = q*a + w*a$,
\item $\left\lbrack a,q+w\right\rbrack_+ = \left\lbrack a,q\right\rbrack_{+} + \left\lbrack a, w \right\rbrack_{+}$,
\item $\left\lbrack a,q\circ w\right\rbrack_+ = \left\lbrack a,q\right\rbrack_{+} + \left\lbrack a, w \right\rbrack_{+}$.
\end{enumerate}
\end{lem}
\begin{proof}
$(1)$ Note that $$a*(q+w) = a*q + q + a*w -q = a*q + \left\lbrack q, a*w \right\rbrack_{+} +a*w.$$ As $a \in \Gamma_{[c-k]}(B)$ and $a*w \in \Gamma_{[c-1]}(B)$, we have that $ \left\lbrack q, a*w \right\rbrack_{+}\in \Gamma_{[c]}=0$ 
and thus  one obtains that  $$a*(q + w) = a*q + a*w,$$ showing the first identity.

$(2)$ It easily is verified that $$a*(q \circ w) = a* (q + q*w + w).$$ Hence, by (1), 
$$ a*(q\circ w)= a*q + a*(q*w) + a*w.$$ As $a*(q*w) \in \Gamma_{[c-k]}(B)*\Gamma_{[k]}(B)=0$, it follows that  $a*(q\circ w)=a*q+a*w$.

$(3)$ Note that $$(q\circ w)*a = q*(w*a)+w*a+q*a.$$
 As $q*(w*a) \in \Gamma_{[c]}(B)$ and $w*a \in Z(B,+)$, it follows that this is equal to $q*a+w*a$.

$(4)$ Because  $a\circ b = a+\lambda_{a}(b)$, we obtain that $$ (q+w)*a = (q \circ \lambda_{\overline{q}}(w))*a.$$ 
Using identity $(3)$, we get that  $(q+w)*a =q*a + \lambda_{\overline{q}}(w)*a$. Applying this again on the last term, we obtain $$\lambda_{\overline{q}}(w)*a = (\overline{q}*w+w)*a = (\bar{q}*w)*a + \lambda_{\overline{q}*w}(w)*a.$$ As $(\bar{q}*w)*a \in \Gamma_{[k]}(B)*\Gamma_{[c-k]}(B)=0$, we obtain that $\lambda_{\overline{q}}(w)*a=\lambda_{\overline{q}*w}(w)*a$. Denoting $q_0=\overline{q}$ and $q_{n+1}=q_n*w$, we obtain inductively that \[
\lambda_{q_n}(w)*a = \lambda_{q_{n+1}}(w)*a.
\]
As $q_n \in B^n$, it follows that $q_c=0$. Hence, we have shown that $$(q+w)*a = q*a + \lambda_{q_c}(w)*a=q*a+w*a.$$

$(5)$ This follows easily from the fact that $[a,w]_+\in Z(B,+)$.

$(6)$ This follows from  $(5)$ and the fact that 
\[
[a,q*w]_+ \in [\Gamma_{[c-k]}(B),\Gamma_{k}(B)]_+=0.
\]
This completes the proof.
\end{proof}

\begin{thm}\label{thm:noetherian}
Let $B$ be an annihilator nilpotent skew brace. If $B$ is finitely generated (as a skew brace), then $B$ has ACC on sub skew braces.
\end{thm}

\begin{proof}
By Theorem \ref{thm:gamma}, we can proceed by induction on the positive integer $c$ such that $\Gamma_{[c]}(B)=0$. If this integer is $1$, then $B$ is a trivial skew brace of abelian type. In particular, the result becomes a well-known theorem on finitely generated abelian groups.

Suppose we have shown that every annihilator nilpotent skew brace $B$ with $\Gamma_{[n]}(B)=0$ has ACC on sub skew braces.
Consider $B$ an annihilator nilpotent skew brace such that $\Gamma_{[n+1]}(B)=0$. 
Thus, $\Gamma_{[n]}(B) \subseteq \Ann(B)$. Let $H$ be a sub skew brace of $B$. 
Then $H/ (H \cap \Gamma_{[n]}(B))$ is finitely generated by applying the 
induction hypothesis on $B/\Gamma_{[n]}(B)$. It remains to show that 
$H \cap \Gamma_{[n]}(B)$ is finitely generated as a skew brace. As $\Gamma_{[n]}(B) \subseteq \Ann(B)$, it is sufficient to show that $\Gamma_{[n]}(B)$ is finitely generated as a skew brace.

Let $1 \leq k \leq n-1$ and consider $\Gamma_{[k]}(B) * \Gamma_{[n-k]}(B)$. Then, by the induction hypothesis, both $\Gamma_{[k]}(B)$ and $\Gamma_{[n-k]}(B)$ are finitely generated and have ACC on sub skew braces modulo $\Gamma_{[n]}(B)$. By Lemma \ref{lem:noetherianannnilp}, the multiplicative group of the factor $\Gamma_{[k]}(B)/\Gamma_{[n]}(B)$ is finitely generated, say by the natural image of the set $A=\left\lbrace a_1,...,a_r\right\rbrace$ and the additive group of the factor $\Gamma_{[n-k]}(B)/\Gamma_{[n]}(B)$ is finitely generated, say by the natural image of the set $Y=\left\lbrace y_1,...,y_s\right\rbrace$. Without loss of generality, we may assume that the set $Y$ (resp. $A$) is closed under taking additive (resp. multiplicative) inverses. Let $\alpha \in \Gamma_{[k]}(B)$ and $\beta \in \Gamma_{[n-k]}(B)$. Then there exist $s_1,s_2 \in \Gamma_{[n]}(B)$,
$a_{j_1},\dots,a_{j_f}\in A$ and 
$y_{l_1},\dots,y_{l_g}\in Y$ 
such that $\alpha = s_1 \circ a_{j_1}\circ\cdots a_{j_f}$ and $\beta = s_2 + y_{l_1}+\cdots+y_{l_g}$ . Note that 
\begin{align*}
\alpha * \beta &= 
(s_1 \circ a_{j_1}\circ\cdots a_{j_f})*(s_2 + y_{l_1}+\cdots+y_{l_g})\\
&=(a_{j_1}\circ\cdots a_{j_f})*(y_{l_1}+\cdots+y_{l_g})
\end{align*}
where the last equality holds as $s_1,s_2 \in \Gamma_{[n]}(B) \subseteq \Ann(B)$. Applying parts (1) and (3) of Lemma \ref{lem:gammaminus2}, this is further equal to a sum of elements of the form $a_{j_i}*y_{l_j}$. This shows that the additive group $\Gamma_{[k]}(B)*\Gamma_{[n-k]}(B)$ is finitely generated.

Completely analogously, applying $(5)$ and $(6)$ of Lemma \ref{lem:gammaminus2}, one sees that the additive group $\left\lbrack\Gamma_{[k]}(B), \Gamma_{[n-k]}(B)\right\rbrack_+$ is finitely generated. As, by definition, $\Gamma_{[n]}(B)$ is the additive subgroup generated by $\Gamma_{[k]}(B)*\Gamma_{[n-k]}(B)$ and $\left\lbrack \Gamma_{[k]}(B),\Gamma_{[n-k]}(B)\right\rbrack_+$ for 
all $k\in\{1,\dots,n-1\}$, it follows that $\Gamma_{[n]}(B)$ is finitely generated.
\end{proof}

\begin{corollary}
Let $B$ be an annihilator nilpotent skew brace. Then the following statements are equivalent:
\begin{enumerate}
    \item The skew brace $B$ is finitely generated (as a skew brace).
    \item The group $(B,+)$ is finitely generated.
    \item The group $(B,\circ)$ is finitely generated.
\end{enumerate}
\end{corollary}

\begin{proof}
    We prove $(1)\implies(2)$ and $(1)\implies(3)$, as the converses are trivial. 
    We proceed by induction on the length of the annihilator series of $B$. 
    If the length is one, then the skew brace is trivial and the claim follows. Let us assume the result 
    holds for length $k\geq1$. Let $B$ be a skew brace with an annihilator series of length $k+1$. 
    By Theorem \ref{thm:noetherian}, $B$ has ACC on sub skew braces. In particular,
    $\Ann(B)$ is finitely generated (as a skew brace). As $\Ann(B)$ is a trivial skew brace, the group 
    $(\Ann(B),+)=(\Ann(B),\circ)$ is finitely generated. The quotient $B/\Ann(B)$ is annihilator nilpotent
    of length $k$ and finitely generated (as a skew brace). Thus the claim follows from the 
    inductive hypothesis.
\end{proof}

\begin{cor}
\label{cor:Hopfian}
Let $B$ be a finitely generated skew brace that is annihilator nilpotent. Then $B$ is Hopfian, i.e. every surjective skew brace endomorphism of $B$ is an isomorphism.
\end{cor}
\begin{proof}
Let $f: B \longrightarrow B$ be a surjective morphism of braces. Then the induced morphism $\tilde{f}:B/B^2 \longrightarrow B/B^2$ is also surjective. As $B$ is a finitely generated skew brace, it follows that $B/B^2$ is a finitely generated nilpotent group. As $B/B^2$ is Hopfian, it follows that $\tilde{f}$ is an isomorphism. Hence, $\ker f$ is contained in $B^2$. Furthermore, as $f$ is surjective, it can be restricted to a surjective skew brace morphism $f_1: B^2 \longrightarrow B^2$. As $B^2$ is finitely generated by Theorem \ref{thm:noetherian}, the result follows by induction.
\end{proof}

Note that Corollary \ref{cor:Hopfian} can also be proven as a direct consequence of the ACC property, similar to standard proofs in group and module theory.

\section{Periodic elements in annihilator nilpotent skew braces}
\label{torsion}

In this section we link the torsion of the additive group with the torsion of the multiplicative group for annihilator nilpotent skew braces. Futhermore, we show that finitely generated annihilator nilpotent skew braces are residually finite. Recall that in a locally nilpotent group (i.e. finitely generated subgroups are nilpotent) the periodic elements form a normal locally finite subgroup. The periodic elements of a group $G$ will be denoted by $T(G)$. For a skew brace $A=(A,+,\circ)$ we denote by $T_{+}(A)$ the torsion elements of the additive group $(A,+)$, and by $T_{\circ}(A)$ the periodic elements of the multiplicative group $(A,\circ)$.

\begin{lem}\label{lem:torsioncentral}
Let $A$ be an annihilator nilpotent skew brace. If 
$T_{+}(\Ann(A)) =\{ 0\}$, then 
$T_{+}(\Ann_{n+1}(A)/\Ann_{n}(A))=\{0\}$ for all $n$.
\end{lem}

\begin{proof}
We prove this by induction on the nilpotence class $m$. 
If $m=1$ then the result is obvious.
So suppose the result holds for all annihilator nilpotent skew braces of class less than $m$. In particular it holds for $A/\Ann(A)$.
If $T_{+}(\Ann_{n+1}(A)/\Ann_{n}(A))\neq \{ 0\}$ for some $n$, then by the assumption $n>0$ and, thus 
by the induction hypothesis $\Ann_{2}(A)$ contains an element $a+\Ann(A)$ that 
is nonzero and periodic in $\Ann_{2}(A)/\Ann(A)$. In particular, $A*a\subseteq\Ann(A)$
and $ka\in\Ann (A)$ for some positive integer $k$. 
By Lemma~\ref{lem:id*} we know that for any $x,y,z\in A$ we have $x*(y+z)  =
x*y +y +x*z -y$.  In particular, 
for any $b\in A$ and positive integer $n$, 
\begin{align*}
b*(na) &=b* ((n-1)a +a)=
b*(n-1)a +(n-1)a +b*a -(n-1)a.    
\end{align*}
As $b*a\in\Ann(A)$, it follows that $b*na = b*(n-1)a + b*a$. Hence, by induction, $b*na =n(b*a)$.
As $ka\in\Ann(A)$ we obtain that $k(b*a)=0$. 

Suppose $A*a\neq \{ 0\}$. Then let  $b\in A$ with $b*a\neq 0$. So, by the previous,  $b*a$ is a non-trivial elements in $T_{+}(\Ann(A))$, a contradiction. 

Hence, $A*a=\{ 0\}$.
Then  $na = a^n$ for all positive integers $n$ and thus also $\lambda_a (a^n) =a^n$. 
By Lemma~\ref{lem:id*} we know that 
\[
(x+y)*z =x*(\lambda_{x}^{-1}(y) *z) +\lambda_{x}^{-1}(y)*z +x*z
\]
for all $x,y,z\in A$.
Hence, for any $b\in A$,
$$(na)*b=(a+a^{n-1})*b = a*(a^{n-1}*b)  +(a^{n-1})*b +a*b.$$
Since $a^{k-1}*b \in \Ann(A)$ this yields
 $$(na)*b=  ((n-1)a)*b + a*b.$$
So, by induction, 
  $$(na)*b=n(a*b),$$
for all positive integers $n$.  
Because $ka\in\Ann(A)$ we get that $k(a*b)=0$.
Consequently, if  $a*b\neq 0$, for some $b\in A$, 
then $T_{+}(\Ann(A)) \neq \{ 0\}$, a contradiction.

Therefore,  $A*a=a*A=0$. In $(A,+)$ we have 
\[
[x+y,z]_+=x+[y,z]_+-x+[x,z]_+
\]
for all $x,y,x\in A$. 
Then 
\begin{align*}
[ka,b]_+&=[(k-1)a+a,b]_+\\
&= (k-1)a+[a,b]_+-(k-1)a+[(k-1)a,b]_+
\end{align*}
for all $b\in A$.
As $\left\lbrack a, b\right\rbrack_+ \in \Ann(B) \subseteq Z(A,+)$, we thus get that 
 $$[ka,b]_+
= [a,b]_+ + [(k-1)a,b]_+ =...=k[a,b]_+.$$ 
Now, if $[a,b]_+ \neq 0$, then $T_+(\Ann(A))$ is non-trivial. 
Hence, $A*a=a*A=0$ and $a \in Z(A,+)$, which implies that $a \in \Ann(A)$, again a contradiction.
\end{proof}

\begin{pro}\label{pro:torsionpart}
Let $A$ be an annihilator nilpotent skew brace. Then
\begin{enumerate}
    \item 
$T_{+}(A)= T_{\circ}(A)$,
    \item $T_{+}(A)$ is an ideal of $A$,
    \item  $T_{+}(A/T_{+}(A))=\{ 0\}$,
    \item  $T_{+}(A)$ is finite if $A$ is finitely generated as a skew brace.
\end{enumerate}
\end{pro}
\begin{proof}
The annihilator nilpotency yields that $(A,+)$ and $(A,\circ)$ are nilpotent groups. Hence $T_{+}(A)$ (respectively $T_{\circ}(A)$) is a normal subgroup of $(A,+)$ (respectively $(A,\circ)$).
Since each $\lambda_a$ is an additive homomorphism of $(A,+)$ we have that $T_{+}(A)$ is a left ideal of $A$. In particular it is a skew subbrace of $A$. In case $A$ is finitely generated as a skew brace, we get from Theorem~\ref{thm:noetherian} that $T_{+}(A)$ is a skew brace with ACC on sub skew braces. 

Since $A$ is annihilator nilpotent, we proceed by induction on the class $n$ of $A$. If $n=1$, then $A=\Ann(A)$ and then (1)--(3) are obvious. Assume that $A$ is finitely generated as a skew brace. Since $A=\Ann(A)$, $T_+(A)$ is a finite abelian group. Hence (4) holds. Assume now that
the result holds for all
annihilator nilpotent skew braces of nilpotence class $k<n$. We first prove (4). Assume that $A$ is finitely generated as a skew brace. By the inductive hypothesis, $T_+(A/\Ann(A))$ is finite. By Theorem \ref{thm:noetherian}, 
$T_+(\Ann(A))$ is a finite ideal of $A$. Since 
\[
T_+(A)/T_+(\Ann(A))\simeq
(T_+(A)+\Ann(A))/\Ann(A)\subseteq 
T_+(A/\Ann(A)),
\]
it follows that $T_+(A)$ is finite. 

Now we prove that $T_+(A)\subseteq T_{\circ}(A)$. Let $a\in T_+(A)$. Then 
the sub brace $B$ generated by $a$ is a finitely generated annihilator nilpotent
skew brace of nilpotence class $\leq n$. Hence $T_+(B)$ is finite. Since $T_+(B)$ 
is a left ideal of $B$, $T_+(B)\subseteq T_\circ(B)$. Since $a\in T_+(B)$, 
$T_+(A)\subseteq T_\circ(A)$. 

We now prove
that $T_+(A)$ is an ideal of $A$. Since
$T_+(A)$ is a left ideal of $A$ and a normal subgroup of $(A,+)$, it is enough to show that $T_+(A)$ is a normal subgroup of $(A,\circ)$. Let $a\in T_+(A)$ and $b\in A$. Since $T_+(A)\subseteq T_\circ(A)$, 
\[
b\circ a\circ \overline{b}\circ\overline{a}\in T_\circ(\Ann_{n-1}(A))=T_+(\Ann_{n-1}(A)).
\]
Since $b\circ a\circ\overline{b}\circ\overline{a}=b\circ a\circ\overline{b}+\lambda_{b\circ a\circ\overline{b}\circ\overline{a}}(-a)\in T_+(A)$,
\[
b\circ a\circ\overline{b}=b\circ a\circ\overline{b}\circ\overline{a}+\lambda_{b\circ a\circ\overline{b}\circ\overline{a}}(a)\in T_+(A).
\]
Hence $T_+(A)$ is an ideal of $A$. Clearly,
$T_+(A/T_+(A))=0$.

Finally, we prove that $T_+(A)=T_\circ(A)$.
To do so, without loss of generality, we may assume that $A$ is finitely generated as a skew brace. Because $T_{+}(A)$ is an ideal of $A$, $T_{+}(A)\subseteq T_{\circ}(A)$ and $T_{+}(A/T_{+}(A))=\{ 0\}$, we also may assume that $T_{+}(A)=\{ 0\}$. Hence, by Lemma~\ref{lem:torsioncentral}, each $\Ann_{n+1}(A)/\Ann_{n}(A)$ is a trivial brace and a torsion-free abelian (multiplicative) group. It follows that $(A,\circ)$ is torsion-free as well, and the claim follows by induction.
\end{proof}

As a corollary, one immediately obtains the following result:

\begin{corollary}
Let $A$ be an annihilator nilpotent skew brace such that $T_{+}(A)=\{ 0\}$. 
For any $a,b\in A$, $a^n=b^n$ or $na=nb$ implies $a=b$.
\end{corollary}

\begin{lem}\label{lem:finiteanngivesfinite}
Let $B$ be an annihilator nilpotent skew brace. If $\Ann(B)$ is finite and $B$ is finitely generated, then $B$ is finite.
\end{lem}

\begin{proof}
Let $B$ be a counterexample with annihilator chain of minimal length. Either $\Ann_2(B)$ is finite or $\Ann_2(B)$ contains an element that is of infinite order modulo $\Ann(B)$. In the first case, $B/\Ann(B)$ is a counterexample with shorter annihilator chain, so this does not occur. Hence, $\Ann_2(B)$ contains an element of infinite order modulo $\Ann(B)$. We will show that this can not happen. Let $x \in \Ann_2(B)$ of infinite order modulo $\Ann(B)$. Denote by $e$ the exponent of  $(\Ann(B),\circ)$. Consider $g \in B$. Then it holds that $[x,g]_{\circ}$ is contained in $\Ann(B)$ and 
\[
[x^e,g]_{\circ} = [x,g]_{\circ}^e = 0.
\]
Since $ex-x^e\in\Ann(B)$ and
$[x,g]_+\in\Ann(B)$ and $g*x\in\Ann(B)$, 
\[
[x^e,g]_+=[ex,g]_+=e[x,g]_+=0\text{ and }
g*(x^e)=g*(ex)=e(g*x)=0.
\]
Hence $x^e\in\Ann(B)$, a contradiction. Therefore
the result follows.
\end{proof}

\begin{thm}
Let $B$ be an annihilator nilpotent skew brace. If $B$ is finitely generated, then $B$ is residually finite.
\end{thm}
\begin{proof}
Let $T$ denote the torsion of $B$. Let $g \in T$. As $\Ann(B)$ is a finitely generated abelian group, it follows that this is equal to $T' \times \Z^l$ for a positive integer $l$ and finite abelian group $T'$. Denote $N_1= \Z^l$ in $\Ann(B)$. As $N_1 \subseteq \Ann(B)$, it holds that $N_1$ is an ideal of $B$. Consider $B/N_1$. Clearly, $g \not \in N_1$, which implies that $g+N_1 \neq N_1$. Suppose we have inductively defined ideals $N_k$, which are torsion-free with $N_{k-1}\subseteq N_k$ and $N_k/N_{k-1} \subseteq \Ann(B/N_{k-1})$. By Lemma \ref{lem:finiteanngivesfinite}, $\Ann(B/N_{k-1})$ is either finite, such that we end the chain, or $\Ann(B/N_{k-1})$ is infinite. Hence, $\Ann(B/N_{k-1})=T_k\times \Z^{l_k}$. Denote $N_k = \Z^{l_k}+N_{k-1}$. The chain $N_k$ ends due to $B$ having ACC on sub skew braces. Hence, there exists $N_k$ such that $\Ann(B/N_k)$ is finite. This shows, by Lemma \ref{lem:finiteanngivesfinite}, that $B/N_k$ is finite. As $g \not\in N_k$, this is a suitable finite image. Hence, we have shown that every torsion element survives in a finite homomorphic image of $B$.

Let $g \in B\setminus T$. Then consider $B'=B/T$. Either $g \in \Ann(B')$ or $g \not\in \Ann(B')$. In the latter case, one can continue by induction on the length of the annihilator chain. In the former case, consider for a positive integer $k>1$ the ideal $G=\langle g^k\rangle\subseteq \Ann(B')$ of $B'$. Consider $B'/G$. Clearly, $g+G \neq G$, but $g+G$ is a torsion element of $B'/G$. Hence, by the first case $g+G$ survives in a finite homomorphic image of $B'/G$, which shows the result.
\end{proof}

\section{Generating sets of annihilator nilpotent skew braces}
\label{generating}

A well-known result of Hirsch states that the commutator subgroup of a nilpotent group $G$ is a set of 
non-generating elements of $G$. We prove an analogue of this result in the context of skew braces that are annihilator nilpotent.

\begin{thm}\label{thm:HirschBrace}
Let $B$ be a skew brace of right nilpotency class $c+1$, i.e. $B^{(c+1)}=0$ and $B^{(c)}\neq 0$. If $A$ is a sub skew brace of $B$ such that $A+B^2=B$ and $A*B\subseteq A$, then $A=B$.
\end{thm}
\begin{proof}
The proof is by induction on $c$. For $c=1$ it is clear. Consider the ideal series $$ 0 \subseteq B^{(c)} \subseteq ... \subseteq B^{(2)} \subseteq B.$$ Then the induction hypothesis can be applied on $B/B^{(c)}$. This proves that $B/B^{(c)} = A/(A \cap B^{(c)})$. Hence, $B = A + B^{(c)}$. 

We now prove that $B*B\subseteq A$; and thus the result follows. To do so, consider arbitrary elements 
$a+b,x+y\in B$ with $a,x\in A$ and $b,y\in B^{(c)}$.
Then, by Lemma~\ref{lem:id*}, 
\begin{align*}
 (a+b)*(x+y) &= (a+b)*x +x + (a*y)-x\\
   &=a*(\lambda_a^{-1}(b) *x)+ \lambda_a^{-1}(b)*x +a*x +x +(a*y)-x
\end{align*}
Since $B*B$ is an ideal of $B$ and $B^{(c)}*B=0$,  
\begin{gather*}
a*(\lambda_a^{-1}(b) *x)+ \lambda_a^{-1}(b)*x=0
\shortintertext{and hence}
(a+b)*(x+y)=(a*x)+x+(a*y)-x.
\end{gather*}
The assumption $A*B\subseteq A$ implies that indeed $B*B\subseteq A$.
\end{proof}

\begin{thm}\label{thm:HirschBraceNilp}
Let $B$ be an annihilator nilpotent skew brace. If $A$ is a sub skew brace of $B$ such that $A+B^2=B$, then $A=B$.
\end{thm}
\begin{proof}
The proof is by induction on the length of the annihilator series. Consider the series $$ 0 \subseteq \Ann(B) \subseteq ... \subseteq \Ann_n(B) = B.$$

Consider the skew brace $B/\Ann(B)$, on which the induction hypothesis applies. This shows that $B = A + \Ann(B)$. In particular, this implies that $$ B*B \subseteq A*A + \Ann(B)*A +A *\Ann(B) + \Ann(B)*\Ann(B) \subseteq A*A.$$
Hence, we obtain that $B=A+B*B \subseteq A$, showing the result.
\end{proof}

\begin{cor}\label{cor:surjectivebrace}
Let $B$ be an annihilator nilpotent skew brace.
Let $f:A\rightarrow B$ be a homomorphism of skew braces and let 
$\overline{f}:A/A^2 \rightarrow B/B^2$ be the natural induced homomorphism of skew braces.
Then $f$ is surjective if and only if  $\overline{f}$ is surjective.
\end{cor}

\begin{proof}
Assume $\overline{f}$ is surjective. Let $b\in B$. Then there exists $a\in A$ so that
 $$ \overline{f} (a+A^2) =f(a)+ B^2 = b+B^2.$$
Hence, $b\in f(A)+B^2$. Since this holds for arbitrary $b\in B$, it follows that $B=f(A)+B^2$. So, by Theorem~\ref{thm:HirschBraceNilp}, $B=f(A)$. This proves the sufficiency of the statement. The necessity is obvious.
\end{proof}

\begin{cor} Let $B$ be an annihilator  nilpotent skew brace.
A subset 
$S\subseteq B$ generates the skew brace  $B$ if and only if $\overline{S}$ (the natural image of $S$ in $B/B^2$) generates the skew brace  $B/B^2$ (and thus it generates $B/B^2$ as an additive group). 
In particular, if $T$ is a subset of $B$ that contains a representative from each $\lambda$-orbit then $B$ is generated (as a skew brace) by the set $T$.
\end{cor}
\begin{proof}
Let $A$ be the skew brace of  $B$ generated by $S$. 
If $\overline{S}$ generates $B/B^2$ then $B=A +B^2$, where $A$ is the skew brace of $B$ generated by $S$. Hence, by Theorem~\ref{thm:HirschBraceNilp}, $B=A$. The converse is obvious.
To prove the final claim it is sufficient to note that $\lambda_b (a)-a =b*a\in B^2$.
\end{proof}

\begin{defn}
Let $B$ be a skew brace and $H$ a sub skew brace. We call (if it exists) the largest sub skew brace $N$ of $B$ such that $H$ is an ideal in $N$, the idealizer of $H$ in $B$. We denote this by $I(H)$. 
\end{defn}

As noted in \href{https://arxiv.org/abs/2310.07474v1}{arXiv:2310.07474}, the idealizer not always exists.

If it exists, then $I(H)$ can also be characterized as the largest sub skew brace $A$ of $B$ such that $A*H \subseteq H$, $H*A \subseteq H$ and $\left[H,A\right]_+ \subseteq H$. These conditions remind of the notions of idealizer within ring theory and the normalizer within group theory.
 The following proposition does not use the idealizer. Using the length of the annihilator chain, the proof is straightforward, hence it is omitted.

\begin{pro}
Let $B$ be a skew brace. 
\begin{enumerate}
\item If $B$ is annihilator nilpotent, then every sub skew brace $A$ has a subideal chain, i.e. there exists a chain of sub skew braces $A=H_1 \subseteq H_2 \subseteq... \subseteq H_n=B $ with $H_i$ an ideal of $H_{i+1}$.
\item If every sub skew brace has a subideal chain, then every proper sub skew brace is an ideal in a strictly larger sub skew brace of $B$. In particular, in this case every maximal sub skew brace is an ideal.
\end{enumerate}
\end{pro}

If $B$ is a skew brace and $S$ is a subset of $B$
we write $(S)$ to denote the ideal of $B$ generated by $S$. 
Recall from \cite{MR4256133} the following result: the intersection of all maximal 
ideals of a skew brace $B$ equals the set of non-generating elements $b$  
of $B$, i.e.
if  the ideal  $(S \cup \{b\}) = B$ then  the ideal  $(S) = B$, for each subset 
$S \subseteq B$.
This set is denoted by $\Rad(B)$ and is called the radical of $B$.
If there do not exist maximal ideals, then, by definition,  $\Rad(B)=B$.
\begin{cor}
Let $B$ be an annihilator nilpotent skew brace. Then every maximal sub brace skew brace is an ideal and  $\Rad(B)$ is the intersection of all maximal sub skew braces.
\end{cor}
Note that, with a similar proof as in the equivalent characterizations of the Frattini subgroup, one can see that the intersection of all maximal sub skew braces coincides with the set of non-generating elements as skew brace. Here, a non-generating element as skew brace is an element $x$ of a skew brace $B$ such that for any subset $S$ of $B$ one has that the sub skew brace generated by $S \cup \left\lbrace x \right\rbrace$ is $B$ implies that the sub skew brace generated by $S$ is $B$.

\subsection*{Acknowledgments} 

The third author is supported in part by OZR3762 
of Vrije Universiteit Brussel. The second author is supported
by Fonds voor Wetenschappelijk Onderzoek (Flanders). We thank the referee for helpful comments and corrections. 

\bibliographystyle{abbrv}
\bibliography{refs}

\def\cprime{$'$}
\begin{thebibliography}{10}

\bibitem{MR4062375}
E.~Acri, R.~Lutowski, and L.~Vendramin.
\newblock Retractability of solutions to the {Y}ang-{B}axter equation and
  {$p$}-nilpotency of skew braces.
\newblock {\em Internat. J. Algebra Comput.}, 30(1):91--115, 2020.

\bibitem{Ba}
R.~J. Baxter.
\newblock Partition function of the eight-vertex lattice model.
\newblock {\em Ann. Physics}, 70(1):193--228, 1972.

\bibitem{BonJed2021}
M.~Bonatto and P.~Jedlička.
\newblock Central nilpotency of skew braces.
\newblock \href{https://arxiv.org/abs/2109.04389}{arXiv:2109.04389}, 2021.

\bibitem{cedo2020various}
F.~Ced{\'o}, E.~Jespers, {\L}.~Kubat, A.~Van~Antwerpen, and C.~Verwimp.
\newblock On various types of nilpotency of the structure monoid and group of a
  set-theoretic solution of the {Yang--Baxter} equation.
\newblock \href{https://arxiv.org/abs/2011.01724}{arXiv:2011.01724}, 2020.

\bibitem{MR2652212}
F.~Ced\'{o}, E.~Jespers, and J.~Okni\'{n}ski.
\newblock Retractability of set theoretic solutions of the {Y}ang-{B}axter
  equation.
\newblock {\em Adv. Math.}, 224(6):2472--2484, 2010.

\bibitem{MR3177933}
F.~Ced\'{o}, E.~Jespers, and J.~Okni\'{n}ski.
\newblock Braces and the {Y}ang-{B}axter equation.
\newblock {\em Comm. Math. Phys.}, 327(1):101--116, 2014.

\bibitem{MR3957824}
F.~Ced\'{o}, A.~Smoktunowicz, and L.~Vendramin.
\newblock Skew left braces of nilpotent type.
\newblock {\em Proc. Lond. Math. Soc. (3)}, 118(6):1367--1392, 2019.

\bibitem{MR1159288}
P.~Diaconis.
\newblock Applications of group representations to statistical problems.
\newblock In {\em Proceedings of the {I}nternational {C}ongress of
  {M}athematicians, {V}ol. {I}, {II} ({K}yoto, 1990)}, pages 1037--1048. Math.
  Soc. Japan, Tokyo, 1991.

\bibitem{MR1183474}
V.~G. Drinfel\cprime~d.
\newblock On some unsolved problems in quantum group theory.
\newblock In {\em Quantum groups ({L}eningrad, 1990)}, volume 1510 of {\em
  Lecture Notes in Math.}, pages 1--8. Springer, Berlin, 1992.

\bibitem{MR1722951}
P.~Etingof, T.~Schedler, and A.~Soloviev.
\newblock Set-theoretical solutions to the quantum {Y}ang-{B}axter equation.
\newblock {\em Duke Math. J.}, 100(2):169--209, 1999.

\bibitem{MR3861714}
T.~Gateva-Ivanova.
\newblock Set-theoretic solutions of the {Y}ang-{B}axter equation, braces and
  symmetric groups.
\newblock {\em Adv. Math.}, 338:649--701, 2018.

\bibitem{MR2885602}
T.~Gateva-Ivanova and P.~Cameron.
\newblock Multipermutation solutions of the {Y}ang-{B}axter equation.
\newblock {\em Comm. Math. Phys.}, 309(3):583--621, 2012.

\bibitem{MR1637256}
T.~Gateva-Ivanova and M.~Van~den Bergh.
\newblock Semigroups of {$I$}-type.
\newblock {\em J. Algebra}, 206(1):97--112, 1998.

\bibitem{MR3647970}
L.~Guarnieri and L.~Vendramin.
\newblock Skew braces and the {Y}ang-{B}axter equation.
\newblock {\em Math. Comp.}, 86(307):2519--2534, 2017.

\bibitem{MR4023387}
E.~Jespers, L.~. Kubat, A.~Van~Antwerpen, and L.~Vendramin.
\newblock Factorizations of skew braces.
\newblock {\em Math. Ann.}, 375(3-4):1649--1663, 2019.

\bibitem{MR4256133}
E.~Jespers, L.~. Kubat, A.~Van~Antwerpen, and L.~Vendramin.
\newblock Radical and weight of skew braces and their applications to structure
  groups of solutions of the {Y}ang-{B}axter equation.
\newblock {\em Adv. Math.}, 385:Paper No. 107767, 20, 2021.

\bibitem{MR1769723}
J.-H. Lu, M.~Yan, and Y.-C. Zhu.
\newblock On the set-theoretical {Y}ang-{B}axter equation.
\newblock {\em Duke Math. J.}, 104(1):1--18, 2000.

\bibitem{MR3935814}
H.~Meng, A.~Ballester-Bolinches, and R.~Esteban-Romero.
\newblock Left braces and the quantum {Y}ang-{B}axter equation.
\newblock {\em Proc. Edinb. Math. Soc. (2)}, 62(2):595--608, 2019.

\bibitem{MR2132760}
W.~Rump.
\newblock A decomposition theorem for square-free unitary solutions of the
  quantum {Y}ang-{B}axter equation.
\newblock {\em Adv. Math.}, 193(1):40--55, 2005.

\bibitem{MR2278047}
W.~Rump.
\newblock Braces, radical rings, and the quantum {Y}ang-{B}axter equation.
\newblock {\em J. Algebra}, 307(1):153--170, 2007.

\bibitem{MR3291816}
W.~Rump.
\newblock The brace of a classical group.
\newblock {\em Note Mat.}, 34(1):115--144, 2014.

\bibitem{MR3814340}
A.~Smoktunowicz.
\newblock On {E}ngel groups, nilpotent groups, rings, braces and the
  {Y}ang-{B}axter equation.
\newblock {\em Trans. Amer. Math. Soc.}, 370(9):6535--6564, 2018.

\bibitem{MR3763907}
A.~Smoktunowicz and L.~Vendramin.
\newblock On skew braces (with an appendix by {N}. {B}yott and {L}.
  {V}endramin).
\newblock {\em J. Comb. Algebra}, 2(1):47--86, 2018.

\bibitem{MR1809284}
A.~Soloviev.
\newblock Non-unitary set-theoretical solutions to the quantum {Y}ang-{B}axter
  equation.
\newblock {\em Math. Res. Lett.}, 7(5-6):577--596, 2000.

\bibitem{Ya}
C.~N. Yang.
\newblock Some exact results for the many-body problem in one dimension with
  repulsive delta-function interaction.
\newblock {\em Phys. Rev. Lett.}, 19(23):1312--1315, 1967.

\end{thebibliography}
\end{document}